\documentclass[11pt]{article}
\usepackage{amsfonts,amscd,amssymb, amsmath}
\newtheorem{definition}{Definition}[section]

\newtheorem{proposition}[definition]{Proposition}

\newtheorem{remark}[definition]{Remark}

\newtheorem{theorem}[definition]{Theorem}
\newtheorem{example}[definition]{Example}

\def\v{\varphi}

\def\mf{\mathfrak}

\newcommand{\nat}{\mbox{$\;\natural \;$}}

\newcommand{\trl}{\mbox{${\;}$$\triangleright \hspace{-1.6mm}<$${\;}$}}
\newcommand{\gsl}{\mbox{${\;}$$>\hspace{-1.7mm}\triangleleft$${\;}$}}

\newcommand{\oot}{\overline{\otimes}}
\def\rawo\lonra{\longrightarrow}

\def\ot{\otimes}

\newcommand{\selabel}[1]{\label{se:#1}}

\allowdisplaybreaks[4]

\newenvironment{proof}{{\it Proof.}}{\hfill $ \square $ \vskip 4mm}

\begin{document}
\title{Two-sided crossed products}
\author {Florin Panaite\\
Institute of Mathematics of the 
Romanian Academy\\ 
PO-Box 1-764, RO-014700 Bucharest, Romania\\
e-mail: Florin.Panaite@imar.ro
}
\date{}
\maketitle

\begin{abstract}
Given two associative algebras $A$, $C$ and a linear space $V$ together with some linear maps $R_1$, $R_2$, $R_3$, $E$ satisfying some conditions, we define an associative algebra structure on $A\ot V\ot C$ called a two-sided crossed product. Particular cases of this construction are the 
iterated twisted tensor product of algebras and the two-sided crossed product over a quasi-bialgebra. 
\end{abstract}

\section*{Introduction}
${\;\;\;}$If $A$ and $B$ are associative unital algebras and $R:B\otimes A\rightarrow 
A\otimes B$ is a linear map satisfying certain conditions (such an $R$ is called a twisting map) 
then $A\otimes B$ becomes an associative unital algebra with a multiplication defined 
in terms of $R$ and the multiplications of $A$ and $B$; this algebra structure is denoted 
by $A\otimes _RB$ and called the twisted tensor product of $A$ and $B$ afforded by $R$ 
(cf. \cite{Cap}, \cite{VanDaele}). Twisted tensor products of algebras appear in various algebraic and geometrical contexts (see for instance \cite{gustafson}, \cite{jlpvo}, \cite{ocal} and references therein). One of their nice features is that they can be iterated: given three twisted tensor products 
$A\otimes _{R_1}B$, $B\otimes _{R_2}C$ and $A\otimes _{R_3}C$, it has been proved in \cite{jlpvo}
that 
a sufficient condition for being able to define certain twisting maps 
$T_1:C\otimes (A\otimes _{R_1}B)\rightarrow (A\otimes _{R_1}B)\otimes C$  and 
$T_2:(B\otimes _{R_2}C)\otimes A\rightarrow A\otimes (B\otimes _{R_2}C)$  associated 
to $R_1$, $R_2$, $R_3$ and ensuring that the algebras $A\otimes _{T_2}(B\otimes _{R_2}C)$ 
and $(A\otimes _{R_1}B)\otimes _{T_1}C$ are equal (this algebra is called the iterated 
twisted tensor product), can be given in terms of the maps $R_1$, $R_2$, $R_3$, 
namely, they have to satisfy the braid relation $(id_A\otimes R_2)\circ 
(R_3\otimes id_B)\circ (id_C\otimes R_1)=(R_1\otimes id_C)\circ (id_B\otimes R_3)\circ 
(R_2\otimes id_A)$. 

An important and substantial generalization of twisted tensor products of algebras is the so called Brzezi\'{n}ski crossed product, introduced in \cite{brz} (it contains as well as particular case  the classical Hopf crossed 
product). 
Given an associative unital algebra $A$, a vector space $V$ endowed with a  
distinguished element $1_V$  and two linear maps $\sigma :V\otimes V
\rightarrow A\otimes V$ and $R:V\otimes A\rightarrow A\otimes V$ 
satisfying certain conditions, Brzezi\'{n}ski's crossed product is a certain 
associative unital algebra structure on $A\otimes V$, denoted in what follows by 
$A\otimes _{R, \sigma }V$.  In \cite{panaite} was introduced a mirror version of the Brzezi\'{n}ski crossed product, denoted by 
$W\overline{\otimes}_{P, \nu }D$ 
(where $D$ is an associative algebra, $W$ is a vector space and $P$, $\nu $ are certain 
maps), and it was proved that under certain circumstances Brzezi\'{n}ski crossed products can be iterated, in the following sense: 
 if $W\overline{\otimes}_{P, \nu }D$ 
and $D\otimes _{R, \sigma }V$ are two crossed products and $Q:V\otimes W\rightarrow 
W\otimes D\otimes V$ is a linear map satisfying certain conditions, then one can define 
certain maps $\overline{\sigma }$, $\overline{R}$, $\overline{\nu }$, $\overline{P}$ such 
that we have the crossed products $(W\overline{\otimes}_{P, \nu }D)\otimes _{\overline{R}, 
\overline{\sigma }}V$ and $W\overline{\otimes }_{\overline{P}, \overline{\nu }}
(D\otimes _{R, \sigma }V)$ that are moreover equal as algebras (this algebra structure is called 
the iterated crossed product). Particular cases of this situation are the iterated twisted tensor products of algebras and the two-sided smash product over a quasi-bialgebra. 

The aim of this paper is to introduce a different type of iterating Brzezi\'{n}ski crossed products. Namely, if $A$ and $C$ are associative unital algebras and $V$ is a linear space endowed with a distinguished nonzero element, and we have linear maps 
$R_1:V\ot A\rightarrow A\ot V$, $R_2:C\ot V\rightarrow V\ot C$, $R_3:C\ot A\rightarrow A\ot C$ and $E:V\ot V\rightarrow A\ot V\ot C$ satisfying some conditions, then we can define a certain  Brzezi\'{n}ski crossed product $A\ot _{R, \sigma }(V\ot C)$, a certain mirror version  Brzezi\'{n}ski crossed product $(A\ot V)\overline{\otimes }_{P, \nu }C$ and we have an algebra isomorphism between them given by the trivial identification. This algebra structure is denoted by $A\gsl V\trl C$ and is called the {\em two-sided crossed product} afforded by the maps $R_1$, $R_2$, $R_3$, $E$. This result admits also a converse. Particular cases are the iterated twisted tensor products of algebras, the so called two-sided crossed product over a quasi-bialgebra from  \cite{bpvo}, \cite{hn},  
and also a recent construction from \cite{ma}.
\section{Preliminaries}\selabel{1}
${\;\;\;\;}$
We work over a commutative field $k$. All algebras, linear spaces
etc. will be over $k$; unadorned $\ot $ means $\ot_k$. By ''algebra'' we 
always mean an associative unital algebra. The multiplication 
of an algebra $A$ is denoted by $\mu _A$ or simply $\mu $ when 
there is no danger of confusion, and we usually denote 
$\mu _A(a\ot a')=aa'$ for all $a, a'\in A$. The unit of $A$ is denoted by $1_A$. 

We recall from \cite{Cap}, \cite{VanDaele} that, given two algebras $A$, $B$ 
and a $k$-linear map $R:B\ot A\rightarrow A\ot B$, with notation 
$R(b\ot a)=a_R\ot b_R$, for $a\in A$, $b\in B$, satisfying the conditions 
$a_R\otimes 1_R=a\otimes 1$, $1_R\otimes b_R=1\otimes b$, 
$(aa')_R\otimes b_R=a_Ra'_r\otimes b_{R_r}$, 
$a_R\otimes (bb')_R=a_{R_r}\otimes b_rb'_R$, 
for all $a, a'\in A$ and $b, b'\in B$ (where $r$ and $R$ are two different indices), 
if we define on $A\ot B$ a new multiplication, by 
$(a\ot b)(a'\ot b')=aa'_R\ot b_Rb'$, then this multiplication is associative 
with unit $1\ot 1$. In this case, the map $R$ is called 
a {\em twisting map} between $A$ and $B$ and the new algebra 
structure on $A\ot B$ is denoted by $A\ot _RB$ and called the 
{\em twisted tensor product} of $A$ and $B$ afforded by the map $R$. 

We recall from \cite{brz} the construction of  
Brzezi\'{n}ski's crossed product:
\begin{proposition} (\cite{brz}) \label{defbrz}
Let $(A, \mu , 1_A)$ be an (associative unital) algebra and $V$ a 
vector space equipped with a distinguished element $1_V\in V$. Then 
the vector space $A\ot V$ is an associative algebra with unit $1_A\ot 1_V$ 
and whose multiplication has the property that $(a\ot 1_V)(b\ot v)=
ab\ot v$, for all $a, b\in A$ and $v\in V$, if and only if there exist 
linear maps $\sigma :V\ot V\rightarrow A\ot V$ and 
$R:V\ot A\rightarrow A\ot V$  satisfying the following conditions:
\begin{eqnarray}
&&R(1_V\ot a)=a\ot 1_V, \;\;\;R(v\ot 1_A)=1_A\ot v, \;\;\;\forall 
\;a\in A, \;v\in V, \label{brz1} \\
&&\sigma (1_V\ot v)=\sigma (v\ot 1_V)=1_A\ot v, \;\;\;\forall 
\;v\in V, \label{brz2} \\
&&R\circ (id_V\ot \mu )=(\mu \ot id_V)\circ (id_A\ot R)\circ (R\ot id_A), 
\label{brz3} \\
&&(\mu \ot id_V)\circ (id_A\ot \sigma )\circ (R\ot id_V)\circ 
(id_V\ot \sigma ) \nonumber \\
&&\;\;\;\;\;\;\;\;\;\;
=(\mu \ot id_V)\circ (id_A\ot \sigma )\circ (\sigma \ot id_V), \label{brz4} \\
&&(\mu \ot id_V)\circ (id_A\ot \sigma )\circ (R\ot id_V)\circ 
(id_V\ot R ) \nonumber \\
&&\;\;\;\;\;\;\;\;\;\;
=(\mu \ot id_V)\circ (id_A\ot R )\circ (\sigma \ot id_A). \label{brz5} 
\end{eqnarray}
If this is the case, the multiplication of $A\ot V$ is given explicitly by
\begin{eqnarray*} 
&&\mu _{A\ot V}=(\mu _2\ot id_V)\circ (id_A\ot id_A\ot \sigma )\circ 
(id_A\ot R\ot id_V),
\end{eqnarray*}
where $\mu _2=\mu \circ (id_A\ot \mu )=\mu \circ (\mu \ot id_A)$. 
We denote by $A\ot _{R, \sigma }V$ this algebra structure and 
call it the {\em crossed product} afforded by the data $(A, V, R, \sigma )$.  
\end{proposition}

If  $A\ot _{R, \sigma }V$ is a crossed product, we use the 
Sweedler-type notation $R(v\ot a)=a_R\ot v_R$, $\sigma (v\ot v')=\sigma _1(v, v') 
\ot \sigma _2(v, v')$, 
for all $v, v'\in V$ and $a\in A$. With this notation, the multiplication of 
 $A\ot _{R, \sigma }V$ reads
\begin{eqnarray*}
&&(a\ot v)(a'\ot v')=aa'_R\sigma _1(v_R, v')\ot \sigma _2(v_R, v'), \;\;\;
\forall \;a, a'\in A, \;v, v'\in V.
\end{eqnarray*}

A twisted tensor product is a particular case of a crossed product 
(cf. \cite{guccione}), namely, if $A\ot _RB$ is a twisted tensor product of 
algebras then $A\ot _RB=A\ot _{R, \sigma }B$, where 
$\sigma :B\ot B\rightarrow A\ot B$ 
is given by $\sigma (b\ot b')=1_A\ot bb'$, for all $b, b'\in B$. 

We recall also from \cite{panaite} the mirror version of Brzezi\'{n}ski's crossed product: 
\begin{theorem}(\cite{panaite})\label{mainmirror}
Let $(B, \mu , 1_B)$ be an (associative unital) algebra and 
$W$ a vector space equipped with a distinguished element 
$1_W\in W$. Then the vector space $W\ot B$ is an associative 
algebra with unit $1_W\ot 1_B$ and whose 
multiplication has the property that 
$(w\ot b)(1_W\ot b')=w\ot bb'$, for all $b, b'\in B$ and 
$w\in W$, if and only if there exist linear maps 
$\nu :W\ot W\rightarrow W\ot B$ and 
$P:B\ot W\rightarrow W\ot B$ satisfying the following conditions: 
\begin{eqnarray}
&&P(b\ot 1_W)=1_W\ot b, \;\;\; P(1_B\ot w)=w\ot 1_B, \;\;\;
\forall \; b\in B, \;w\in W, \label{mirtwunit} \\
&&\nu (w\ot 1_W)=\nu (1_W\ot w)=w\ot 1_B, \;\;\;
\forall \;w\in W, \label{mircocunit} \\
&&P\circ (\mu \ot id_W)=(id_W\ot \mu )\circ 
(P\ot id_B)\circ (id_B\ot P), \label{mirtwmap}\\
&&(id_W\ot \mu )\circ (\nu \ot id_B)\circ 
(id_W\ot P)\circ (\nu \ot id_W)\nonumber \\
&&\;\;\;\;\;\;\;\;\;\;
=(id_W\ot \mu )\circ (\nu \ot id_B)\circ (id_W\ot \nu ), 
\label{mir1}\\
&&(id_W\ot \mu )\circ (\nu \ot id_B)\circ (id_W\ot P)
\circ (P\ot id_W)\nonumber \\
&&\;\;\;\;\;\;\;\;\;\;
=(id_W\ot \mu )\circ (P\ot id_B)\circ (id_B\ot \nu ).
\label{mir2}
\end{eqnarray}
If this is the case, the multiplication of $W\ot B$ is given 
explicitely by 
\begin{eqnarray*}
&&\mu _{W\ot B}=(id_W\ot \mu _2)\circ (\nu \ot id_B
\ot id_B)\circ (id_W\ot P\ot id_B), 
\end{eqnarray*}
where $\mu _2=\mu \circ (id_B\ot \mu )=
\mu \circ (\mu \ot id_B)$. We denote by 
$W\overline{\ot}_{P, \nu }B$ this algebra structure 
and call it the {\em crossed product} 
afforded by the data $(W, B, P, \nu )$. 
\end{theorem}

If $W\oot _{P, \nu }B$ is a crossed product, we use also the 
Sweedler-type notation $P(b\ot w)=w_P\ot b_P$, $\nu (w\ot w')=
\nu _1(w, w')\ot \nu _2(w, w')$, 
for all $w, w'\in W$ and $b\in B$. With this notation, 
the multiplication of $W\oot _{P, \nu }B$ reads 
\begin{eqnarray*}
&&(w\ot b)(w'\ot b')=\nu _1(w, w'_P)\ot 
\nu _2(w, w'_P)b_Pb', \;\;\;\forall \;b, b'\in B, \;w, w'\in W. 
\end{eqnarray*}

If $A\ot _RB$ is a twisted tensor product product of algebras, 
then $A\ot _RB=A\oot _{R, \nu }B$, where 
$\nu :A\ot A\rightarrow A\ot B$, 
$\nu (a\ot a')=aa'\ot 1_B$ for all $a, a'\in A$. 

\section{The definition of the two-sided crossed product}\selabel{2}
\setcounter{equation}{0}
\begin{theorem}\label{main}
Let $(A, \mu _A,  1_A)$ and $(C, \mu _C, 1_C)$ be two associative unital algebras and $V$ a linear space equipped with a distinguished nonzero element $1_V\in V$. Assume that we are given linear maps (with respective notation) 
\begin{eqnarray*}
&&R_1:V\ot A\rightarrow A\ot V, \;\;\; R_1(v\ot a)=a_{R_1}\ot v_{R_1}=a_{r_1}\ot v_{r_1}, \\
&&R_2:C\ot V\rightarrow V\ot C, \;\;\; R_2(c\ot v)=v_{R_2}\ot c_{R_2}=v_{r_2}\ot c_{r_2}, \\
&&R_3:C\ot A\rightarrow A\ot C,  \;\;\; R_3(c\ot a)=a_{R_3}\ot c_{R_3}=a_{r_3}\ot c_{r_3}, \\
&&E:V\ot V\rightarrow A\ot V\ot C, \;\;\; E(v\ot v')=E_A(v, v')\ot E_V(v, v')\ot E_C(v, v'), 
\end{eqnarray*}
for all $a\in A$, $v, v'\in V$, $c\in C$,  such that the following conditions are satisfied: \\[2mm]
(i) $R_3$ is a twisting map between $A$ and $C$, that is, for all $a, a'\in A$ and $c, c'\in C$:
\begin{eqnarray}
&&R_3(c\ot 1_A)=1_A\ot c, \;\;\; R_3(1_C\ot a)=a\ot 1_C, \label{twR31}\\
&&(aa')_{R_3}\ot c_{R_3}=a_{R_3}a'_{r_3}\ot (c_{R_3})_{r_3}, \label{twR32}\\
&&a_{R_3}\ot (cc')_{R_3}=(a_{R_3})_{r_3}\ot c_{r_3}c'_{R_3}; \label{twR33}
\end{eqnarray}
(ii) $R_1(1_V\ot a)=a\ot 1_V$, $R_1(v\ot 1_A)=1_A\ot v$, for all $a\in A$, $v\in V$; \\[2mm]
(iii) $R_2(c\ot 1_V)=1_V\ot c$, $R_2(1_C\ot v)=v\ot 1_C$, for all $v\in V$, $c\in C$; \\[2mm]
(iv) $E(1_V\ot v)=E(v\ot 1_V)=1_A\ot v\ot 1_C$, for all $v\in V$; \\[2mm]
(v) the following relations are satisfied: 
\begin{eqnarray}
&&R_1\circ (id_V\ot \mu _A)=(\mu _A\ot id_V)\circ (id_A\ot R_1)\circ (R_1\ot id_A), \label{brz3R1} \\
&&R_2\circ (\mu _C\ot id_V)=(id_V\ot \mu _C)\circ (R_2\ot id_C)\circ (id_C\ot R_2), \label{brz3R2} \\
&&(id_A\ot R_2)\circ (R_3\ot id_V)\circ (id_C\ot R_1)=(R_1\ot id_C)\circ (id_V\ot R_3)\circ (R_2\ot id_A), \label{braidV} \\
&&(\mu _A\ot id_V\ot id_C)\circ (id_A\ot E)\circ (R_1\ot id_V)\circ (id_V\ot R_1) \nonumber \\
&&\;\;\;\;\;\;\;\;=(\mu _A\ot id_V\ot id_C)\circ (id_A\ot R_1\ot id_C)\circ (id_A\ot id_V\ot R_3)\circ (E\ot id_A), \label{octoR1} \\
&&(id_A\ot id_V\ot \mu _C)\circ (E\ot id_C)\circ (id_V\ot R_2)\circ (R_2\ot id_V) \nonumber \\
&&\;\;\;\;\;\;\;\; =(id_A\ot id_V\ot \mu _C)\circ (id_A\ot R_2\ot id_C)\circ (R_3\ot id_V\ot id_C)\circ (id_C\ot E), \label{octoR2} \\
&&(\mu_A\ot id_V\ot \mu _C)\circ (id_A\ot E\ot id_C)\circ (R_1\ot id_V\ot id_C)\circ (id_V\ot E)\nonumber \\
&&\;\;\;\;\;\;\;\; =(\mu _A\ot id_V\ot \mu _C)\circ (id_A\ot E\ot id_C)\circ (id_A\ot id_V\ot R_2)\circ (E\ot id_V). \label{octoR1R2}
\end{eqnarray}
Define the linear maps 
\begin{eqnarray*}
&&R:(V\ot C)\ot A\rightarrow A\ot (V\ot C), \;\;\; R=(R_1\ot id_C)\circ (id_V\ot R_3), \\
&&P:C\ot (A\ot V)\rightarrow (A\ot V)\ot C, \;\;\; P=(id_A\ot R_2)\circ (R_3\ot id_V), \\
&&\sigma :(V\ot C)\ot (V\ot C)\rightarrow A\ot (V\ot C), \\
&& \sigma =(id_A\ot id_V\ot \mu _C)\circ (E\ot \mu _C)\circ (id_V\ot R_2\ot id_C), \\
&&\nu :(A\ot V)\ot (A\ot V)\rightarrow (A\ot V)\ot C, \\
&&\nu =(\mu _A\ot id_V\ot id_C)\circ (\mu _A\ot E)\circ (id_A\ot R_1\ot id_V), 
\end{eqnarray*}
that is 
\begin{eqnarray*}
&&R((v\ot c)\ot a)=(a_{R_3})_{R_1}\ot (v_{R_1}\ot c_{R_3}), \\
&&P(c\ot (a\ot v))=(a_{R_3}\ot v_{R_2})\ot (c_{R_3})_{R_2}, \\
&&\sigma ((v\ot c)\ot (v'\ot c'))=E_A(v, v'_{R_2}) \ot (E_V(v, v'_{R_2})\ot E_C(v, v'_{R_2})c_{R_2}c'), \\
&&\nu ((a\ot v)\ot (a'\ot v'))=(aa'_{R_1}E_A(v_{R_1}, v')\ot E_V(v_{R_1}, v'))\ot E_C(v_{R_1}, v'), 
\end{eqnarray*}
for all $a, a'\in A$, $v, v'\in V$, $c, c'\in C$. Then we have the crossed products $A\ot _{R, \sigma }(V\ot C)$ and 
$(A\ot V)\overline{\otimes }_{P, \nu }C$ and an algebra isomorphism between them given by the trivial identification.  

This associative unital algebra structure on $A\ot V\ot C$ will be denoted by 
$A\gsl V\trl C$ and will be called a {\em two-sided crossed product} (afforded by the maps $R_1$, $R_2$, $R_3$, $E$). Its unit is $1_A\ot 1_V\ot 1_C$ and its multiplication is (for all $a, a'\in A$, $v, v'\in V$, $c, c'\in C$)
\begin{eqnarray*}
&&(a\ot v\ot c)(a'\ot v'\ot c')=a(a'_{R_3})_{R_1}E_A(v_{R_1}, v'_{R_2})\ot E_V(v_{R_1}, v'_{R_2}) \ot
E_C(v_{R_1}, v'_{R_2})(c_{R_3})_{R_2}c'.
\end{eqnarray*}
Sometimes we will denote $a\gsl v\trl c$ instead of $a\ot v\ot c$. 
\end{theorem}
\begin{proof}
Note first that the relations (\ref{brz3R1})-(\ref{octoR1R2}) are respectively equivalent (by writing them down on elements) to the following relations (for all $a, a'\in A$, $v, v', v''\in V$, $c, c'\in C$):
\begin{eqnarray}
&&(aa')_{R_1}\ot v_{R_1}=a_{R_1}a'_{r_1}\ot (v_{R_1})_{r_1},  \label{equiv1} \\
&&v_{R_2}\ot (cc')_{R_2}=(v_{R_2})_{r_2}\ot c_{r_2}c'_{R_2}, \label{equiv2} \\
&&(a_{R_1})_{R_3}\ot (v_{R_1})_{R_2}\ot (c_{R_3})_{R_2}=(a_{R_3})_{R_1}\ot (v_{R_2})_{R_1}\ot (c_{R_2})_{R_3}, 
\label{equiv3} \\
&&(a_{R_1})_{r_1}E_A(v_{r_1}, v'_{R_1})\ot E_V(v_{r_1}, v'_{R_1})\ot E_C(v_{r_1}, v'_{R_1}) \nonumber \\
&&\;\;\;\;\;\;\;\; =E_A(v, v')(a_{R_3})_{R_1}\ot E_V(v, v')_{R_1}\ot E_C(v, v')_{R_3}, \label{equiv4} \\
&&E_A(v_{R_2}, v'_{r_2})\ot E_V(v_{R_2}, v'_{r_2})\ot E_C(v_{R_2}, v'_{r_2})(c_{R_2})_{r_2} \nonumber \\
&&\;\;\;\;\;\;\;\; =E_A(v, v')_{R_3}\ot E_V(v, v')_{R_2}\ot (c_{R_3})_{R_2}E_C(v, v'), \label{equiv5} \\
&&E_A(v', v'')_{R_1}E_A(v_{R_1}, E_V(v', v''))
\ot E_V(v_{R_1}, E_V(v', v'')) \nonumber \\
&&\;\;\;\;\;\;\;\;\;\;\;\;\;\;\;\;\;\;\;\ot E_C(v_{R_1}, E_V(v', v''))
E_C(v', v'') \nonumber \\
&&\;\;\;\;\;\;\;\; =E_A(v, v')E_A(E_V(v, v'), v''_{R_2})\ot E_V(E_V(v, v'), v''_{R_2}) \nonumber \\
&&\;\;\;\;\;\;\;\;\;\;\;\;\;\;\;\;\;\;\;  \ot  E_C(E_V(v, v'), v''_{R_2})E_C(v, v')_{R_2}. \label{equiv6}
\end{eqnarray}

We prove now that we have the crossed product $A\ot _{R, \sigma }(V\ot C)$, the proof for 
$(A\ot V)\overline{\otimes }_{P, \nu }C$ is similar and left to the reader. The conditions (\ref{brz1}) and (\ref{brz2}) are very easy to prove, so we need to prove (\ref{brz3})-(\ref{brz5}). We denote by $R_2=r_2=\mathcal{R}_2=\overline{R}_2$ and $R_3=r_3=\overline{R}_3$ some more copies of $R_2$ and $R_3$. \\[2mm]
\underline{Proof of (\ref{brz3})}:\\[2mm]
${\;\;\;\;}$$(\mu _A\ot id_V\ot id_C)\circ (id_A\ot R)\circ (R\ot id_A)(v\ot c\ot a\ot a')$
\begin{eqnarray*}
&=&(\mu _A\ot id_V\ot id_C)\circ (id_A\ot R)((a_{R_3})_{R_1}\ot v_{R_1}\ot c_{R_3}\ot a')\\
&=&(a_{R_3})_{R_1}(a'_{R_3})_{r_1}\ot (v_{R_1})_{r_1}\ot (c_{R_3})_{r_3}\\
&\overset{(\ref{equiv1})}{=}&(a_{R_3}a'_{r_3})_{R_1}\ot v_{R_1}\ot (c_{R_3})_{r_3}\\
&\overset{(\ref{twR32})}{=}&((aa')_{R_3})_{R_1}\ot v_{R_1}\ot c_{R_3}=R(v\ot c\ot aa')\\
&=&R\circ (id_V\ot id_C\ot \mu _A)(v\ot c\ot a\ot a'), \;\;\; q.e.d.
\end{eqnarray*}
\underline{Proof of (\ref{brz4})}:\\[2mm]
$(\mu _A\ot id_V\ot id_C)\circ (id_A\ot \sigma )\circ (R\ot id_V\ot id_C)\circ (id_V\ot id_C\ot \sigma )
(v\ot c\ot v'\ot c'\ot v''\ot c'')$
\begin{eqnarray*}
&=&(\mu _A\ot id_V\ot id_C)\circ (id_A\ot \sigma )\circ (R\ot id_V\ot id_C)(v\ot c\ot E_A(v', v''_{R_2})\\
&&\ot E_V(v', v''_{R_2})
\ot E_C(v', v''_{R_2})c'_{R_2}c'')\\
&=&(\mu _A\ot id_V\ot id_C)\circ (id_A\ot \sigma )((E_A(v', v''_{R_2})_{R_3})_{R_1}\ot v_{R_1}\ot c_{R_3}\\
&&\ot 
E_V(v', v''_{R_2})\ot E_C(v', v''_{R_2})c'_{R_2}c'')\\
&=&(E_A(v', v''_{R_2})_{R_3})_{R_1}E_A(v_{R_1}, E_V(v', v''_{R_2})_{r_2})\ot E_V(v_{R_1}, E_V(v', v''_{R_2})_{r_2})\\
&&\ot E_C(v_{R_1}, E_V(v', v''_{R_2})_{r_2})(c_{R_3})_{r_2}E_C(v', v''_{R_2})c'_{R_2}c''\\
&\overset{(\ref{equiv5})}{=}&E_A(v'_{\mathcal{R}_2}, (v''_{R_2})_{\overline{R}_2})_{R_1}
E_A(v_{R_1}, E_V(v'_{\mathcal{R}_2}, (v''_{R_2})_{\overline{R}_2}))\ot 
E_V(v_{R_1}, E_V(v'_{\mathcal{R}_2}, (v''_{R_2})_{\overline{R}_2}))\\
&& \ot E_C(v_{R_1}, E_V(v'_{\mathcal{R}_2}, (v''_{R_2})_{\overline{R}_2}))E_C(v'_{\mathcal{R}_2}, (v''_{R_2})_{\overline{R}_2})(c_{\mathcal{R}_2})_{\overline{R}_2}c'_{R_2}c''\\
&\overset{(\ref{equiv6})}{=}&E_A(v, v'_{\mathcal{R}_2})E_A(E_V(v, v'_{\mathcal{R}_2}), ((v''_{R_2})_{\overline{R}_2})_{r_2})
\ot E_V(E_V(v, v'_{\mathcal{R}_2}), ((v''_{R_2})_{\overline{R}_2})_{r_2})\ot \\
&&E_C(E_V(v, v'_{\mathcal{R}_2}), ((v''_{R_2})_{\overline{R}_2})_{r_2})E_C(v, v'_{\mathcal{R}_2})_{r_2}
(c_{\mathcal{R}_2})_{\overline{R}_2}c'_{R_2}c''\\
&\overset{(\ref{equiv2})}{=}&E_A(v, v'_{\mathcal{R}_2})E_A(E_V(v, v'_{\mathcal{R}_2}), v''_{R_2})
\ot E_V(E_V(v, v'_{\mathcal{R}_2}), v''_{R_2})\ot \\
&&E_C(E_V(v, v'_{\mathcal{R}_2}), v''_{R_2})\{E_C(v, v'_{\mathcal{R}_2})
c_{\mathcal{R}_2}c'\}_{R_2}c''\\
&=&(\mu _A\ot id_V\ot id_C)\circ (id_A\ot \sigma )\circ (\sigma \ot id_V\ot id_C) (v\ot c\ot v'\ot c'\ot v''\ot c''), \;\;\; q.e.d.
\end{eqnarray*}
\underline{Proof of (\ref{brz5})}:\\[2mm]
${\;\;\;}$$(\mu _A\ot id_V\ot id_C)\circ (id_A\ot \sigma )\circ (R\ot id_V\ot id_C)\circ (id_V\ot id_C\ot R)(v\ot c\ot v'\ot c'\ot a')$
\begin{eqnarray*}
&=&(\mu _A\ot id_V\ot id_C)\circ (id_A\ot \sigma )((((a'_{R_3})_{R_1})_{r_3})_{r_1}\ot v_{r_1}\ot c_{r_3}\ot v'_{R_1}\ot c'_{R_3})\\
&=&(((a'_{R_3})_{R_1})_{r_3})_{r_1}E_A(v_{r_1}, (v'_{R_1})_{R_2})\ot E_V(v_{r_1}, (v'_{R_1})_{R_2})\ot 
E_C(v_{r_1}, (v'_{R_1})_{R_2})(c_{r_3})_{R_2}c'_{R_3}\\
&\overset{(\ref{equiv3})}{=}&(((a'_{R_3})_{r_3})_{R_1})_{r_1}E_A(v_{r_1}, (v'_{R_2})_{R_1})\ot E_V(v_{r_1}, (v'_{R_2})_{R_1})\ot 
E_C(v_{r_1}, (v'_{R_2})_{R_1})(c_{R_2})_{r_3}c'_{R_3}\\
&\overset{(\ref{equiv4})}{=}&E_A(v, v'_{R_2})
(((a'_{R_3})_{r_3})_{\overline{R}_3})_{R_1}\ot E_V(v, v'_{R_2})_{R_1}\ot E_C(v, v'_{R_2})_{\overline{R}_3}
(c_{R_2})_{r_3}c'_{R_3}\\
&\overset{(\ref{twR33})}{=}&E_A(v, v'_{R_2})(a'_{R_3})_{R_1}\ot  E_V(v, v'_{R_2})_{R_1}\ot 
\{E_C(v, v'_{R_2})
c_{R_2}c'\}_{R_3}\\
&=&(\mu _A\ot id_V\ot id_C)\circ (id_A\ot R)\circ (\sigma \ot id_A)(v\ot c\ot v'\ot c'\ot a'), \;\;\; q.e.d.
\end{eqnarray*}

The only thing left to prove is that the multiplications of the two crossed products coincide. The multiplication of 
$A\ot _{R, \sigma }(V\ot C)$ reads:\\[2mm]
${\;\;\;\;\;}$
$(a\ot (v\ot c))(a'\ot (v'\ot c'))$
\begin{eqnarray*}
&=&aa'_{R}\sigma _1((v\ot c)_R, v'\ot c')\ot \sigma _2((v\ot c)_R, v'\ot c')\\
&=&a(a'_{R_3})_{R_1}\sigma _1(v_{R_1}\ot c_{R_3}, v'\ot c')\ot \sigma _2(v_{R_1}\ot c_{R_3}, v'\ot c')\\
&=&a(a'_{R_3})_{R_1}E_A(v_{R_1}, v'_{R_2})\ot E_V(v_{R_1}, v'_{R_2})\ot E_C(v_{R_1}, v'_{R_2})
(c_{R_3})_{R_2}c'. 
\end{eqnarray*}
The multiplication of $(A\ot V)\overline{\otimes }_{P, \nu }C$ reads: \\[2mm]
${\;\;\;\;\;}$$((a\ot v)\ot c)((a'\ot v')\ot c')$
\begin{eqnarray*}
&=&\nu _1(a\ot v, (a'\ot v')_P)\ot \nu _2(a\ot v, (a'\ot v')_P)c_Pc'\\
&=&\nu _1(a\ot v, a'_{R_3}\ot v'_{R_2})\ot \nu _2(a\ot v,  a'_{R_3}\ot v'_{R_2})(c_{R_3})_{R_2}c'\\
&=&a(a'_{R_3})_{R_1}E_A(v_{R_1}, v'_{R_2})\ot E_V(v_{R_1}, v'_{R_2})\ot E_C(v_{R_1}, v'_{R_2})
(c_{R_3})_{R_2}c',
\end{eqnarray*}
and we can see that the two multiplications are defined by the same formula. 
\end{proof}

Theorem \ref{main} admits the following converse. 
\begin{theorem}\label{converse}
Let $(A, \mu _A,  1_A)$ and $(C, \mu _C, 1_C)$ be two associative unital algebras and $V$ a linear space equipped with a distinguished nonzero element $1_V\in V$. Assume that on $A\ot V\ot C$ we have an associative algebra structure (with multiplication denoted by $\cdot $) with unit $1_A\ot 1_V\ot 1_C$ such that the maps 
\begin{eqnarray*}
&&A\rightarrow A\ot V\ot C, \;\;\; a\mapsto a\ot 1_V\ot 1_C, \\
&&C\rightarrow A\ot V\ot C, \;\;\; c\mapsto 1_A\ot 1_V\ot c
\end{eqnarray*}
are algebra maps, and moreover 
\begin{eqnarray}
&&(1_A\ot v\ot 1_C)\cdot (a'\ot 1_V\ot 1_C)=\sum _ia_i\ot v_i\ot 1_C, \label{ajut1} \\
&&(1_A\ot 1_V\ot c)\cdot (1_A\ot v'\ot 1_C)=\sum _j1_A\ot v_j\ot c_j, \label{ajut2} \\
&&(1_A\ot 1_V\ot c)\cdot (a'\ot 1_V\ot 1_C)=\sum _la_l\ot 1_V\ot c_l, \label{ajut3} \\
&&a\ot v\ot c=(a\ot 1_V\ot 1_C)\cdot (1_A\ot v\ot 1_C)\cdot (1_A\ot 1_V\ot c), \label{ajut4}
\end{eqnarray}
for all $a, a'\in A$, $v, v'\in V$, $c, c'\in C$, where $a_i, a_l\in A$, $v_i, v_j\in V$, $c_j, c_l\in C$ are some elements. Then there exist linear maps $R_1$, $R_2$, $R_3$, $E$ satisfying the conditions in Theorem \ref{main} and the given algebra structure 
on $A\ot V\ot C$ coincides with the two-sided crossed product $A\gsl V\trl C$  afforded by the maps $R_1$, $R_2$, $R_3$, $E$.
\end{theorem}
\begin{proof}
We regard the linear spaces $A$, $V$, $C$, $A\ot V$, $V\ot C$, $A\ot C$ embedded canonically into $A\ot V\ot C$, so we can 
regard elements $a\in A$, $v\in V$, $c\in C$ as elements in $A\ot V\ot C$; by (\ref{ajut4}) we can write 
$a\ot v\ot c=a\cdot v\cdot c$, for all $a\in A$, $v\in V$, $c\in C$. We define the linear map $E:V\ot V\rightarrow A\ot V\ot C$ by 
$E(v\ot v')=(1_A\ot v\ot 1_C)\cdot (1_A\ot v'\ot 1_C)$, that is $E(v\ot v')=v\cdot v'$. Also, because of (\ref{ajut1}), 
(\ref{ajut2}), (\ref{ajut3}), we can define the linear maps $R_1:V\ot A\rightarrow A\ot V$, $R_2:C\ot V\rightarrow V\ot C$, 
$R_3:C\ot A\rightarrow A\ot C$, such that $R_1(v\ot a)=v\cdot a$, $R_2(c\ot v)=c\cdot v$, $R_3(c\ot a)=c\cdot a$, 
for all $a\in A$, $v\in V$, $c\in C$. With this notation, in $A\ot V\ot C$ we have $v\cdot a=a_{R_1}\cdot v_{R_1}$, 
$c\cdot v=v_{R_2}\cdot c_{R_2}$, $c\cdot a=a_{R_1}\cdot c_{R_1}$, for all $a\in A$, $v\in V$, $c\in C$. Also, if we denote 
$E(v\ot v')=E_A(v, v')\ot E_V(v, v')\ot E_C(v, v')$, then we have 
$v\cdot v'=E_A(v, v')\cdot E_V(v, v')\cdot E_C(v, v')$.

First we need to prove that we have indeed a two-sided crossed product $A\gsl V\trl C$. 
We prove the relation (\ref{equiv4}). Let $a\in A$ and $v, v'\in V$. In the algebra $(A\ot V\ot C, \cdot )$ we have 
$v\cdot (v'\cdot a)=(v\cdot v')\cdot a$. We compute: 
\begin{eqnarray*}
v\cdot (v'\cdot a)&=&v\cdot (a_{R_1}\cdot v'_{R_1})=(v\cdot a_{R_1})\cdot v'_{R_1}\\
&=&((a_{R_1})_{r_1}\cdot v_{r_1})\cdot v'_{R_1}=(a_{R_1})_{r_1}\cdot (v_{r_1}\cdot v'_{R_1})\\
&=&(a_{R_1})_{r_1}\cdot E_A(v_{r_1}, v'_{R_1})\cdot E_V(v_{r_1}, v'_{R_1})\cdot E_C(v_{r_1}, v'_{R_1})\\
&=&((a_{R_1})_{r_1}E_A(v_{r_1}, v'_{R_1}))\cdot E_V(v_{r_1}, v'_{R_1})\cdot E_C(v_{r_1}, v'_{R_1})\\
&=&(a_{R_1})_{r_1}E_A(v_{r_1}, v'_{R_1}))\ot E_V(v_{r_1}, v'_{R_1})\ot E_C(v_{r_1}, v'_{R_1}),
\end{eqnarray*}
\begin{eqnarray*}
(v\cdot v')\cdot a&=&E_A(v, v')\cdot E_V(v, v')\cdot E_C(v, v')\cdot a\\
&=&E_A(v, v')\cdot E_V(v, v')\cdot a_{R_3} \cdot E_C(v, v')_{R_3}\\
&=&E_A(v, v')\cdot (a_{R_3})_{R_1}\cdot  E_V(v, v')_{R_1} \cdot E_C(v, v')_{R_3}\\
&=&(E_A(v, v')(a_{R_3})_{R_1})\cdot  E_V(v, v')_{R_1} \cdot E_C(v, v')_{R_3}\\
&=&E_A(v, v')(a_{R_3})_{R_1}\ot  E_V(v, v')_{R_1} \ot E_C(v, v')_{R_3},
\end{eqnarray*}
so (\ref{equiv4}) is satisfied. 
The other relations in Theorem \ref{main} are proved in a similar way, for instance for (\ref{equiv1}) one uses $v\cdot (aa')=(v\cdot a)\cdot a'$, 
for (\ref{equiv2}) one uses $(cc')\cdot v=c\cdot (c'\cdot v)$, for (\ref{equiv3}) one uses $c\cdot (v\cdot a)=(c\cdot v)\cdot a$, 
for (\ref{equiv5}) one uses $(c\cdot v)\cdot v'=c\cdot (v\cdot v')$, and for (\ref{equiv6}) one uses 
$(v\cdot v')\cdot v''=v\cdot (v'\cdot v'')$, for some elements $a, a'\in A$, $v, v', v''\in V$ and $c, c'\in C$.

Thus, the hypotheses of Theorem \ref{main} are satisfied, so we can consider the two-sided crossed product 
afforded by the maps $R_1$, $R_2$, $R_3$, $E$. The only thing left to prove is that the original multiplication $\cdot $ 
of $A\ot V\ot C$ coincides with the one of the two-sided crossed product. For $a, a'\in A$, $v, v'\in V$ and $c, c'\in C$ 
we compute:\\[2mm]
${\;\;\;\;\;\;}$$(a\ot v\ot c)\cdot (a'\ot v'\ot c')$
\begin{eqnarray*}
&=&a\cdot v\cdot c\cdot a'\cdot v'\cdot c'=a\cdot v\cdot a'_{R_3}\cdot c_{R_3}\cdot v'\cdot c'\\
&=&a\cdot (a'_{R_3})_{R_1}\cdot v_{R_1}\cdot v'_{R_2}\cdot (c_{R_3})_{R_2}\cdot c'\\
&=&a(a'_{R_3})_{R_1}\cdot E_A(v_{R_1}, v'_{R_2})\cdot E_V(v_{R_1}, v'_{R_2})\cdot E_C(v_{R_1}, v'_{R_2})
\cdot (c_{R_3})_{R_2}\cdot c'\\
&=&a(a'_{R_3})_{R_1}E_A(v_{R_1}, v'_{R_2})\cdot E_V(v_{R_1}, v'_{R_2})\cdot E_C(v_{R_1}, v'_{R_2})
(c_{R_3})_{R_2}c'\\
&=&a(a'_{R_3})_{R_1}E_A(v_{R_1}, v'_{R_2})\ot E_V(v_{R_1}, v'_{R_2})\ot E_C(v_{R_1}, v'_{R_2})
(c_{R_3})_{R_2}c', 
\end{eqnarray*}
and this is indeed the multiplication of $A\gsl V\trl C$. 
\end{proof}

Two-sided crossed products have the following universal property, generalizing the universal property of iterated twisted tensor products of algebras (\cite{jlpvo}, Theorem 2.7): 
\begin{theorem}
Let $A\gsl V\trl C$ be a two-sided crossed product afforded by the maps $R_1$, $R_2$, $R_3$, $E$, let $(X, \mu _X, 1_X)$ 
be an associative unital algebra, and $f_A:A\rightarrow X$, $f_V:V\rightarrow X$, $f_C:C\rightarrow X$ linear maps such that 
$f_A$ and $f_C$ are unital algebra maps and the following conditions are satisfied: 
\begin{eqnarray*}
&&\mu _X\circ (f_C\ot f_V\ot f_A)=\mu _X\circ (f_A\ot f_V\ot f_C)\circ (id_A\ot R_2)\circ (R_3\ot id_B)\circ (id_C\ot R_1), \\
&&\mu _X\circ (f_A\ot f_V\ot f_C)\circ E=\mu _X\circ (f_V\ot f_V).
\end{eqnarray*}
Then there exists a unique algebra map $f:A\gsl V\trl C\rightarrow X$ such that $f(a\gsl 1_V\trl 1_C)=f_A(a)$, 
$f(1_A\gsl v\trl 1_C)=f_V(v)$ and $f(1_A\gsl 1_V\trl c)=f_C(c)$, for all $a\in A$, $v\in V$, $c\in C$. 
\end{theorem}
\begin{proof}
Similar to the one is \cite{jlpvo}; $f$ is uniquely defined by $f(a\gsl v\trl c)=f_A(a)f_V(v)f_C(c)$, details are left to the reader. 
\end{proof}

We consider now two particular situations. 
\begin{remark} {\em 
Let $A\gsl V\trl C$ be a two-sided crossed product afforded by the maps $R_1$, $R_2$, $R_3$, $E$ and assume that $R_1$ is the flip map, $R_1(v\ot a)=a\ot v$ for all $a\in A$, $v\in V$. Through the canonical linear isomorphism $A\ot V\ot C\simeq V\ot A\ot C$ we transfer the algebra structure of $A\gsl V\trl C$ to $V\ot (A\ot C)$, where it becomes 
\begin{eqnarray*}
&&(v\ot (a\ot c))*(v'\ot (a'\ot c'))=E_V(v, v'_{R_2})\ot (aa'_{R_3}E_A(v, v'_{R_2})\ot E_C(v, v'_{R_2})(c_{R_3})_{R_2}c').
\end{eqnarray*}
Define the linear maps (for all $a\in A$, $v, v'\in V$, $c\in C$)
\begin{eqnarray*}
&&P:(A\ot _{R_3}C)\ot V\rightarrow V\ot (A\ot _{R_3}C), \;\;\; P((a\ot c)\ot v)=v_{R_2}\ot (a\ot c_{R_2}), \\
&&\nu :V\ot V\rightarrow V\ot (A\ot _{R_3}C), \;\;\;
\nu (v\ot v')=E_V(v, v')\ot (E_A(v, v') \ot E_C(v, v')).
\end{eqnarray*}
 Then, since we have $(v\ot (a\ot c))*(1_V\ot (a'\ot c'))=v\ot (a\ot c)\cdot (a'\ot c')$, for all $a, a'\in A$, $v\in V$, $c, c'\in C$, 
 where the multiplication $\cdot $ is the 
one of the twisted tensor product $A\ot _{R_3}C$, it turns out that $*$ is just the multiplication of the crossed product $V\overline{\otimes}_{P, \nu }(A\ot _{R_3}C)$. }
\end{remark}
\begin{remark} \label{rem2} {\em 
Let again $A\gsl V\trl C$ be a two-sided crossed product afforded by the maps $R_1$, $R_2$, $R_3$, $E$ and assume this time that $R_3$ is the flip map, $R_3(c\ot a)=a\ot c$ for all $a\in A$, $c\in C$. Consider again the canonical linear isomorphism $A\ot V\ot C\simeq V\ot A\ot C$ and transfer the algebra structure of $A\gsl V\trl C$ to $V\ot (A\ot C)$, where it becomes 
\begin{eqnarray*}
&&(v\ot (a\ot c))\bullet (v'\ot (a'\ot c'))=E_V(v_{R_1}, v'_{R_2})\ot (aa'_{R_1}E_A(v_{R_1}, v'_{R_2})\ot E_C(v_{R_1}, v'_{R_2})c_{R_2}c').
\end{eqnarray*}
For $a, a'\in A$, $v\in V$, $c, c'\in C$ we have $(v\ot (a\ot c))\bullet (1_V\ot (a'\ot c'))=v_{R_1}\ot (aa'_{R_1}\ot cc')$; in general, this is different from $v\ot (a\ot c)(a'\ot c')$, therefore in general the multiplication $\bullet $ cannot be written as the multiplication of some crossed product $V\overline{\otimes}_{P, \nu }(A\ot C)$ (here $A\ot C$ is the ordinary tensor product algebra between $A$ and $C$). 

Define the linear maps (for all $a\in A$, $v, v'\in V$, $c\in C$)
\begin{eqnarray*}
&&J:(A\ot C)\ot V\rightarrow V\ot (A\ot C), \;\;\; J((a\ot c)\ot v)=v_{R_2}\ot (a\ot c_{R_2}), \\
&&T:V\ot (A\ot C)\rightarrow V\ot (A\ot C), \;\;\; T(v\ot (a\ot c))=v_{R_1}\ot (a_{R_1}\ot c), \\
&&\gamma :V\ot V\rightarrow V\ot V\ot (A\ot C), \;\;\; \gamma (v\ot v')=v\ot v'\ot (1_A\ot 1_C), \\
&&\eta :V\ot V\rightarrow V\ot (A\ot C)\ot (A\ot C), \\
&& \eta (v\ot v')=E_V(v, v')\ot (1_A\ot E_C(v, v'))\ot (E_A(v, v')\ot 1_C). 
\end{eqnarray*}
Then one can check that these maps satisfy the conditions (1.1)-(1.13) in \cite{pan}, so we have the L-R-crossed product 
$V\ot _{J, T, \gamma , \eta }(A\ot C)$ (again $A\ot C$ is the ordinary tensor product algebra between $A$ and $C$). Its multiplication is (with notation as in \cite{pan})\\[2mm]
${\;\;\;\;}$$(v\ot (a\ot c))(v'\ot (a'\ot c'))$
\begin{eqnarray*}
&=&\eta _1(\gamma _1(v, v')_T, \gamma _2(v, v')_J)\\[1mm]
&&\;\;\;\;\;\;\;\;\;\;\;\;\;\;\ot 
\eta _2(\gamma _1(v, v')_T, \gamma _2(v, v')_J)(a\ot c)_J\gamma _3(v, v')(a'\ot c')_T
\eta _3(\gamma _1(v, v')_T, \gamma _2(v, v')_J)\\
&=&\eta _1(v_T, v'_J)\ot \eta _2(v_T, v'_J)(a\ot c)_J(a'\ot c')_T\eta _3(v_T, v'_J)\\
&=&\eta _1(v_{R_1}, v'_{R_2})\ot \eta _2(v_{R_1}, v'_{R_2})(a\ot c_{R_2})(a'_{R_1}\ot c')\eta _3(v_{R_1}, v'_{R_2})\\
&=&E_V(v_{R_1}, v'_{R_2})\ot (1_A\ot E_C(v_{R_1}, v'_{R_2}))(a\ot c_{R_2})(a'_{R_1}\ot c')
(E_A(v_{R_1}, v'_{R_2})\ot 1_C)\\
&=&E_V(v_{R_1}, v'_{R_2})\ot (aa'_{R_1}E_A(v_{R_1}, v'_{R_2})\ot E_C(v_{R_1}, v'_{R_2})c_{R_2}c'), 
\end{eqnarray*}
thus the multiplication $\bullet $ is the multiplication of the L-R-crossed product $V\ot _{J, T, \gamma , \eta }(A\ot C)$.
}
\end{remark}
\section{Examples}\selabel{3}
\setcounter{equation}{0}
\begin{example} {\em 
We recall from \cite{jlpvo} the construction of iterated twisted tensor 
products of algebras. Let $A$, $B$, $C$ be associative unital algebras, 
$R_1:B\ot A\rightarrow A\ot B$, $R_2:C\ot B\rightarrow B\ot C$, 
$R_3:C\ot A\rightarrow A\ot C$ twisting maps  satisfying the braid (or hexagon) 
equation 
\begin{eqnarray*}
&&(id_A\ot R_2)\circ (R_3\ot id_B)\circ (id_C\ot R_1)=
(R_1\ot id_C)\circ (id_B\ot R_3)\circ (R_2\ot id_A). 
\end{eqnarray*}
Then we have an algebra structure on $A\ot B\ot C$ (called the iterated twisted 
tensor product) with unit $1_A\ot 1_B\ot 1_C$ and multiplication 
\begin{eqnarray}
&&(a\ot b\ot c)(a'\ot b'\ot c')=a(a'_{R_3})_{R_1}\ot b_{R_1}b'_{R_2}\ot 
(c_{R_3})_{R_2}c'. \label{multiterttp}
\end{eqnarray}
Define the linear mp $E:B\ot B\rightarrow A\ot B\ot C$, $E(b\ot b')=1_A\ot bb'\ot 1_C$, for all $b, b'\in B$. Then one can easily check that we have a two-sided crossed product $A\gsl B\trl C$ afforded by the maps $R_1$, $R_2$, $R_3$, $E$ and its multiplication is given by (\ref{multiterttp}). So, the iterated twisted tensor product of algebras is a particular case of the two-sided crossed product. 
}
\end{example}
\begin{example} {\em 
We recall from \cite{hn}, \cite{bpvo} the construction of the two-sided crossed product over a quasi-bialgebra (we use notation as in \cite{bpvo}). Let $H$ be 
a quasi-bialgebra, $\mathfrak{A}$ a right $H$-comodule algebra, $\mathfrak{B}$ a 
left $H$-comodule algebra and $\mathcal{A}$ an $H$-bimodule algebra. Then  
$\mathfrak{A}\ot \mathcal{A}\ot \mathfrak{B}$ becomes an associative unital algebra, with unit $1_{\mf {A}}\ot 1_{{\cal A}} \ot 1_{\mf {B}}$
and multiplication 
\begin{eqnarray}
&&\hspace*{-1.5cm} 
(\mathfrak{a}\ot \varphi \ot
\mathfrak{b})(\mathfrak{a}'\ot \varphi ' \ot \mathfrak{b}')\nonumber\\
&=&\mathfrak{a}\mathfrak{a}'_{<0>}\tilde{x} ^1_{\rho }\ot (\tilde{x}
^1_{\lambda }\cdot \varphi \cdot   
\mathfrak{a}'_{<1>}\tilde{x} ^2_{\rho })(\tilde{x} ^2_{\lambda }\mathfrak{b}_{[-1]}\cdot  \varphi ' 
\cdot \tilde{x} ^3_{\rho })\ot \tilde{x} ^3_{\lambda }\mathfrak{b}_{[0]}\mathfrak{b}' ,\label{gtscp}
\end{eqnarray}
for all $\mf {a}, \mf {a}'\in \mf {A}$, $\mf {b}, \mf
{b}'\in \mf {B}$ and $\v, \v ' \in {\cal A}$.  One can easily check that the hypotheses of Theorem \ref{converse} are satisfied, so we obtain from Theorem \ref{converse} linear maps $R_1$, $R_2$, $R_3$, $E$ which can be easily seen to be defined by the following formulae:
\begin{eqnarray*}
&&R_1:\mathcal{A}\ot \mathfrak{A}\rightarrow \mathfrak{A}\ot \mathcal{A}, \;\;\;R_1(\varphi \ot \mathfrak{a})=\mathfrak{a}_{<0>}\ot \varphi \cdot \mathfrak{a}_{<1>}, \\
&&R_2:\mathfrak{B}\ot \mathcal{A}\rightarrow \mathcal{A}\ot \mathfrak{B}, \;\;\;R_2(\mathfrak{b}\ot \varphi )=b_{[-1]}\cdot \varphi \ot b_{[0]}, \\
&&R_3:\mathfrak{B}\ot \mathfrak{A}\rightarrow \mathfrak{A}\ot \mathfrak{B}, \;\;\;R_3(\mathfrak{b}\ot \mathfrak{a})=\mathfrak{a}\ot \mathfrak{b}, \\
&&E:\mathcal{A}\ot \mathcal{A}\rightarrow \mathfrak{A}\ot \mathcal{A}\ot \mathfrak{B}, \;\;\;
E(\varphi \ot \varphi ')=\tilde{x} ^1_{\rho }\ot (\tilde{x}
^1_{\lambda }\cdot \varphi \cdot   
\tilde{x} ^2_{\rho })(\tilde{x} ^2_{\lambda }\cdot  \varphi ' 
\cdot \tilde{x} ^3_{\rho })\ot \tilde{x} ^3_{\lambda },
\end{eqnarray*}
for all $\mf {a}\in \mf {A}$, $\v, \v ' \in {\cal A}$ and 
$\mf {b}\in \mf {B}$, and moreover the two-sided crossed product over $H$ coincides with the two-sided crossed product (in the sense of this paper) 
$\mf {A}\gsl {\cal A}\trl \mf {B}$ afforded by the maps $R_1$, $R_2$, $R_3$, $E$. 

Note that $R_3$ is the flip map, so, by Remark \ref{rem2}, $\mf {A}\gsl {\cal A}\trl \mf {B}$ is isomorphic to an L-R-crossed product 
$\mathcal{A}\nat _{J, T, \gamma , \eta }(\mathfrak{A}\ot \mathfrak{B})$. On the other hand, we know from \cite{pvo}, Proposition 2.5, that $\mf {A}\gsl {\cal A}\trl \mf {B}$ is isomorphic to the L-R-smash product $\mathcal{A}\nat (\mathfrak{A}\ot \mathfrak{B})$. 
So, Remark \ref{rem2} is actually a generalization of \cite{pvo}, Proposition 2.5.  
}
\end{example}
\begin{example} {\em 
We recall the following construction from \cite{ma}, called as well a two-sided crossed product. Let $(H, \Delta , \varepsilon )$ be a coassociative counital coalgebra endowed with a distinguished element $1_H\in H$ suct that $\Delta (1_H)=1_H\ot 1_H$, let 
$(A, \mu _A, 1_A)$ and $(B, \mu _B, 1_B)$ be two associative unital algebras and assume that we are given linear maps 
\begin{eqnarray*}
&&G:H\ot H\rightarrow A\ot H, \;\;\; h\ot h'\mapsto h^G\ot h'_G, \\
&&R:H\ot A\rightarrow A\ot H, \;\;\; h\ot a\mapsto a_R\ot h_R, \\
&&T:B\ot H\rightarrow H\ot B, \;\;\;b\ot h\mapsto h_T\ot b_T, \\
&& \tau :H\ot H\rightarrow B, 
\end{eqnarray*}
such that several conditions are satisfied. Then $A\ot H\ot B$ becomes an associative algebra with unit $1_A\ot 1_H\ot 1_B$ and multiplication 
\begin{eqnarray*}
&&(a\ot h\ot b)(a'\ot h'\ot b')=aa'_R((h_R)_1)^G\ot ((h'_T)_1)_G\ot \tau ((h_R)_2, (h'_T)_2)b_Tb'. 
\end{eqnarray*}
Then one can check that the hypotheses of Theorem \ref{converse} are satisfied, so we obtain the linear maps $R_1=R$, $R_2=T$, $R_3=$ flip map, $R_3(b\ot a)=a\ot b$, for all $a\in A$, $b\in B$, and 
\begin{eqnarray*}
E:H\ot H\rightarrow A\ot H\ot B, \;\;\; E(h\ot h')=(h_1)^G\ot (h'_1)_G\ot \tau (h_2, h'_2), \;\;\; \forall \; h, h'\in H. 
\end{eqnarray*}
So, the two-sided crossed product in \cite{ma} coincides with the two-sided crossed product $A\gsl H\trl B$ afforded by the maps 
$R_1$, $R_2$, $R_3$, $E$.
}
\end{example}

\begin{center}
ACKNOWLEDGEMENTS
\end{center}
The author was partially supported by a grant from UEFISCDI,
project number PN-III-P4-PCE-2021-0282.


\end{document}